\pdfoutput=1
\documentclass[utf8,12pt,english,natbib,upmath]{aclart}

\usepackage{aclraccourcis,hyperref}

\usepackage{xcolor}
\definecolor[named]{lipicsLightGray}{rgb}{0.85,0.85,0.86}

\definecolor{keywordcolor}{rgb}{0.7, 0.1, 0.1}   
\definecolor{tacticcolor}{rgb}{0.0, 0.1, 0.3}    
\definecolor{commentcolor}{rgb}{0.4, 0.4, 0.4}   
\definecolor{stringcolor}{rgb}{0.5, 0.3, 0.2}    
\definecolor{symbolcolor}{rgb}{0.1, 0.2, 0.7}    
\definecolor{sortcolor}{rgb}{0.1, 0.5, 0.1}      
\definecolor{attributecolor}{rgb}{0.7, 0.1, 0.1} 
\definecolor{errorcolor}{rgb}{1, 0, 0}           

\usepackage{listings}
\usepackage{upgreek}

\title[An experiment in formalization of an intermediate-level algebra theorem]{Formalizing the proof of an intermediate-level algebra theorem
--- An experiment} 
\author{Antoine Chambert-Loir}
\address{Université Paris Cité \\
 Institut de mathématiques de Jussieu --- Paris Rive Gauche \\
8 place Aurélie Nemours \\
75013 Paris}
\email{antoine.chambert-loir@u-paris.fr}

\begin{abstract}
Proof assistants are computer softwares that allow us to write mathematical
proofs so as to assess their correctness. 
In November~2021, I started the project  of checking 
the simplicity of the alternating groups within the \texttt{Lean}
theorem prover
and its \texttt{mathlib} library. This text aims at reviewing
this experiment.
\end{abstract}

\begin{altabstract}
Les assistants de preuves sont des logiciels qui permettent
de rédiger des démonstrations mathématiques 
et d'en garantir leur correction. 
En novembre 2021, j'ai débuté un projet de vérification
de la simplicité des groupes alternés au sein de l'assistant
de preuve \texttt{Lean},
et de sa librairie \texttt{mathlib}.
Ce texte est un essai de compte rendu de cette expérience.
\end{altabstract}

\begin{document}

\maketitle

\section{Introduction}

Human mathematics is written in  plain language, and we all know examples
of shortcomings that lead to “proofs” of wrong results. 
We also know for now more than
hundred years ago, notably by the works of \cite{Peano-1889}
or \cite{WhiteheadRussell-1927},
that mathematics can be written using axiomatic systems,
and, at least in principle,
in a rigid syntactic way, so as to avoid such problems,
at least if the chosen axiomatic system does not lead to contradiction.
I write “in principle”, because this rigid syntactic writing is extremly
verbose: It took hundreds of pages to Whitehead and Russell
to prove that $1+1=2$.
One may find a pleasant, large audience, account of this quest
in the comic book \cite{DoxiadisPapadimitriou-2009}.

Since the 1950s, the development of computers led mathematicians
to propose to use their mechanical force to develop
fully formalized proofs of the mathematical corpus.
Among such examples, let us mention N.~G. De Bruijn's \texttt{Automath}
(1967), A. Trybulec's \texttt{Mizar} (1973), 
G. Huet's team project \texttt{Coq} (1989), 
C. Coquand's \texttt{Agda} (1999) or L. de Moura's \texttt{Lean} (2013)…

In recent years, these softwares have allowed us to check delicate
parts of the mathematical corpus:
Appel and Haken's proof the Four color theorem
(the regions delimited by a finite planar graph can be colored  
in four colors so that any two neighboring regions
have different colors);
Feit and Thompson's proof of the  Odd order theorem
(any finite group of odd cardinality is solvable),
by~\cite{Gonthier-2008}
and~\cite{GonthierAspertiAvigadEtAl-2013},
(both in \texttt{Coq});
Hales's proof  of the Kepler conjecture
(the standard, “cubic close”, sphere packing is the densest one), 
by~\cite{HalesAdamsBauerEtAl-2017} (in \texttt{HOL Light});
following the challenge of~\citet{Scholze-2022},
the proof of a delicate homological algebra result of Clausen and Scholze
(in \texttt{Lean}, the so called “Liquid tensor experiment”, 2022, by \textsc{Commelin} and \textsc{Topaz}, with the help of many more people involved);
or Gromov's proof of the h-principle and the sphere eversion theorem 
by \citet*{MassotvanDoornNash-2022}, 
also in \texttt{Lean}.

Actually, the latter results were not formalized in plain \texttt{Lean},
but were built on the Lean mathematical library \texttt{mathlib}.
Led by a group of approximately 25 people,
plus some 15 reviewers, 
this mathematical library is an ongrowing effort 
of roughly 300 people, with (as today) approximately 45\,000 definitions
and 110\,000 mathematical statements (“theorems”) that cover many
fields of mathematics, such as additive combinatorics, 
complex variables, differential geometry and Lebesgue integration…
So that a collective effort is at all possible, the initial
authors of~\texttt{mathlib} had to make careful 
architecture and design decisions, 
described in~\citep{ThemathlibCommunity-2020}.
As \texttt{Lean}/\texttt{mathlib}
 is an open source project, 
it is also relatively easy to install it
on one's own computer, and start joining this collective effort.
This is also facilitated by a comprehensive website and
an online discussion board where contributors share
their problems and, remarkably generously, insights.

In November~2021, 
I embarked in checking in \texttt{Lean}/\texttt{mathlib}
the proof that the alternating group of a finite set of cardinality
at least~5 is a simple group.
While this  mathematical result is of a  smaller scale,
compared to the above-quoted accomplishements, 
it belongs to the classical (under)graduate mathematical corpus,
and I felt interesting to experiment the formalization process on
a result of this intermediate level. 
For reasons I will try to share, I chose a nonstandard way to do that,
that led me to unsuspected mathematical territories.

This text is a retrospective account of this journey.

\bigskip

\emph{
I thank Javier Fresán for having inviting me to write this paper,
a Spanish version of which appeared in
 \emph{La Gaceta de la Real Sociedad Matemática Española},
Vol. 26 (2023), Núm. 3, Págs. 535–553.
I also thank him, together with Riccardo Brasca, Patrick Massot and Filippo Nuccio,
for their comments, as well as Martin Liebeck and Raphaël Rouquier for their help.
I am grateful to the \emph{mathlib} community for their 
enthusiasm in welcoming newcomers to the game,
and for the support they provide so generously.}

\section{Solvability, simplicity}

Let us first recall the terms of the statement we have in mind.

\begin{theo}\label{theo.an-simple}
Let $n$ be an integer such that $n\geq 5$. 
The alternating group $\mathfrak A_n$ is a \emph{simple} group.
\end{theo}
(For $n\leq 2$, the group $\mathfrak A_n$ is trivial;
for $n=3$, it is cyclic of order~3; 
for $n=4$, it is a nonabelian solvable group of order~12,
its derived subgroup is an abelian subgroup of index~3.)

In the \texttt{Lean} language, this theorem can be formulated as follows:

\begin{lstlisting}[caption={Simplicity of the alternating groups
of order at least five
% \mllink\mllink{number_theory/well_approximable.lean\#L191}
},label=lst:simple,captionpos=t,float=htb]
theorem alternating_group.normal_subgroups {α : Type*}
  [decidable_eq α] [fintype α]
  (hα : 5 ≤ fintype.card α)
  {N : subgroup (alternating_group α)}
  (hnN : N.normal) (ntN : nontrivial N) : N = ⊤
\end{lstlisting}

The command \lstinline!theorem! initiates a statement of a theorem,
followed by the name given to it, here \lstinline!alternating_group.normal_subgroups!, and followed by a sequence of arguments which
are surrounded by various kinds of parentheses.

The first of these arguments, \lstinline!α!, is declared as a \emph{type},
the basic notion  of “dependent type theory”,
the formal language of~\texttt{Lean}: 
in this case, one may think of \lstinline!α! as a set.
The next arguments impose that it is also finite, and 
 \lstinline!hα! is the assumption that it has at least five members.
The next three parameters are \lstinline!N!, which is declared as a subgroup
of the alternating group on \lstinline!α!, \lstinline!hnN! which imposes that 
it is a normal subgroup, and \lstinline!ntN! that it is nontrivial
(which, in the \texttt{mathlib} library, means that
it is not reduced to the unit element).

The conclusion of that theorem follows the colon symbol: \lstinline!N = ⊤!,
meaning that \lstinline!N! is the full group of the alternating group.
In the actual code, this 5-line text is followed by the symbol \lstinline!:=! 
and the actual \emph{proof} of this statement.

Every object of type theory is a type,
and what \texttt{Lean} does is allowing the user to write down
new types, or members of those types. For example, in 
the example above \lstinline!N! is a member 
of the type \lstinline!subgroup (alternating_group α)!. 
\texttt{Lean} provides a few
basic means to define new types from older; for example,
if \lstinline!α! and \lstinline!β! are types, there is a type
\lstinline!α → β! which represents “functions” from \lstinline!α!
to \lstinline!β!, in the sense that if \lstinline!f : α → β!
is such a function and \lstinline!a : α! (read: “\lstinline!a! is a member
of the type \lstinline!α!”), then \lstinline!f a!
is a member of the type~\lstinline!β!, with obvious rules
regarding equality. Functions of multiple arguments can be defined
“à la Curry”: for example, if \lstinline!α!, \lstinline!β! and \lstinline!γ!
are types, then \lstinline!f : α → β → γ! maps \lstinline!a : α!
to \lstinline!f a : β → γ!, which maps \lstinline!b : β!
to \lstinline!f a b : γ!, etc.
Even the expression \lstinline!N = ⊤! of listing~\ref{lst:simple}
designates a type,
namely, the type of proofs of equality between the two members 
\lstinline!N!  and \lstinline!⊤! of the type \lstinline!subgroup (alternating_group α)!,
and the (omitted) code that follows constructs a member of that type,
that is, a proof of that statement:
type theory puts mathematical structures and theorems on the same level.

\medskip

Simple groups are those (nontrivial) groups
whose only normal subgroups are the two 
obvious examples, the full group and trivial group~$\{e\}$.
When a nontrivial group~$G$ is not simple, it admits
a normal subgroup~$H$ such that $H\neq \{e\}$ and $H\neq G$;
then one can (try to) study~$G$ through its projection to
the quotient group~$G/H$, whose kernel is~$H$.
When we restrict ourselves to finite groups, a full “dévissage”
is possible and a common metaphor presents finite simple groups 
as the “elementary particles” of finite group theory.
In this direction, a legendary theorem whose proof 
involved hundreds of mathematicians and hundreds of
mathematical papers written over a period of 50 years,
is the classification of finite simple groups: All finite simple
groups appear in a list of groups of the following form:
\begin{itemize}
\item The cyclic group $\Z/p\Z$, for some prime number~$p$;
\item The alternating group~$\mathfrak A_n$, for some integer $n\geq 5$;
\item Lists of finite groups “of Lie type”, related
to linear algebra over finite fields, whose easiest
examples consist of the projective special linear groups
$\mathrm{PSL}(n,\F_q)$ over a finite field of cardinality $q$,
assuming $q\geq 4$ if $n=2$ ($\mathrm{PSL}(2,\F_2)$
and $\mathrm{PSL}(2,\F_3)$ are respectively isomorphic
to $\mathfrak S_3$ and $\mathfrak A_4$, hence are not simple);
\item A list of~26 (so called “sporadic”) groups, related
to exceptional combinatorial geometries, 
such as the Matthieu groups~$\mathrm M_{11}$, $\mathrm M_{12}$,
$\mathrm M_{22}$, $\mathrm M_{23}$ and $\mathrm M_{24}$.
\end{itemize}

The difficult part of the classification of finite simple groups
 asserts that those finite groups are the only simple groups,
but we are only concerned here by the easy part of the classification,
that these groups are indeed simple.

The first ones, cyclic groups of prime order, are simple:
it follows from Lagrange's theorem (the order of a subgroup
divides the order of the group)
that they have no other subgroup than themselves and the trivial
subgroup.

As an aside remark, let us note that the center~$Z(G)$
of a group~$G$, the set of elements~$g\in G$ which commute
with any other element of~$G$, is also a normal subgroup.
Consequently, if $G$ is simple, then either $Z(G)=G$,
in which case $G$ is commutative, hence  a cyclic group of prime order, 
or $Z(G)=\{e\}$.
This explains why, from the second item on,
all groups of the above list have trivial center.

On the second item of that list come the alternating groups, 
which are the very subject
of this note, and whose simplicity is often established
in algebra lectures related to Galois theory
and the solvability of algebraic equations (in one variable).
While Abel and Ruffini had proved that 
general algebraic equations of degree $\geq 5$ cannot
be solved by radicals, Galois's theorem refines
that result by proving that a given algebraic equation
is solvable by radicals if and only if its Galois group
is solvable. The notion of a group of an equation was introduced 
by Galois, as well as the notion of a normal subgroup
and of solvable group, although
he did not give a name to these two concepts: 
the Galois group is the subgroup of the permutations
of the roots that preserve all algebraic relations
with rational coefficients;
and a finite group~$G$ is solvable if it is
trivial or if, by induction, it admits
a nontrivial normal subgroup~$H$ which is itself solvable
and such that the quotient group~$G/H$ is commutative.
In modern terms, we say that a group~$G$ is solvable
if its “derived series” $G, D(G), D(D(G))\dots$,
the decreasing sequence of subgroups
obtained by successively taking commutator subgroups,
eventually reaches the identity subgroup.

In that perspective, the Abel--Ruffini theorem boils down
to the fact that a general equation of degree~$n$
has Galois group the full symmetric group~$\mathfrak S_n$, 
and that, for $n\geq 5$, it is not solvable,
itself a direct consequence of the following more precise result.

\begin{prop}
Let $n$ be an integer such that $n\geq 5$.
The commutator subgroup of $\mathfrak S_n$ 
is the alternating group $\mathfrak A_n$.
The commutator subgroup of the alternating group~$\mathfrak A_n$
is itself.
\end{prop}
\begin{proof}
Any commutator has signature~$1$, so that $D(\mathfrak S_n)\subseteq \mathfrak A_n$. On the other hand, a commutator of two transpositions
$(a\,b)$ and $(c\,d)$ is trivial if they are equal or have disjoint supports,
but is equal to a 3-cycle otherwise, as the following computation shows
\[ (a\,b)(c\,a)(a\,b)(c\,a)=(a\, b\,c). \]
We see that any 3-cycle can we written as a commutator, so that
$D(\mathfrak S_n)$ contains all 3-cycles, 
which are known to generate the alternating group~$\mathfrak A_n$.
(This works for $n\geq 3$.)

To prove that $D(\mathfrak A_n)$ is $\mathfrak A_n$ itself,
we prove that the quotient group $K=\mathfrak A_n/D(\mathfrak A_n)$
is trivial. The group~$\mathfrak A_n$ is generated by 3-cycles~$g$,
so their images generate~$K$.
The hypothesis $n\geq 5$ implies that all 3-cycles
are conjugate in~$\mathfrak A_n$; consequently,
they all have the same image in~$K$, say~$k$, and $K=\langle k\rangle$.
Since the square of a 3-cycle $g=(a\,b\,c)$ is again a 3-cycle,
namely $g^2=(a\,c\,b)$,
one has $k=k^2$, hence $k=e$ and $K=\{e\}$.
\end{proof}

The relation with simplicity is that 
noncommutative solvable groups cannot be simple.
In fact, it is an elementary observation
that the commutator subgroup $D(G)$ of any group~$G$ is a 
normal subgroup of~$G$; if $G$ is simple,
then either $D(G)=\{e\}$, which means that $G$ is commutative,
or $D(G)=G$.
So Galois's theorem on algebraic equations of degree~$\geq 5$
is often subsumed in mentioning that the alternating group~$\mathfrak A_n$
is simple for $n\geq 5$, 
although the result which is needed is the easier proposition given above.

It is sometimes written that Galois proved that simplicity theorem,
although the only explicit statement I could find in his works is
the fact that the smallest possible cardinality of a simple (noncommutative)
finite group
(he says “indecomposable”) is $5\cdot 4\cdot 3$, 
but he does not state that it corresponds to the
alternating group~$\mathfrak A_5$.  
On the other hand, group theorists of the 19th century, from Lagrange and
Ruffini to Jordan, gradually built the tools to understand
Galois's theorem in terms of the simplicity of the alternating group.

There are many relatively easy proofs of the simplicity of~$\mathfrak A_n$
for $n\geq 5$, such as, for example, 
the one given by~\cite[p.~247]{Jacobson-1985},
but none of them looks as being completely straightforward,
in the sense that they do not tell \emph{why} they work.
Moreover, some of them build on case disjunctions,
or mental reasonings which, although they are quite familiar to us,
remain a bit awkward to specify exactly, to the point
that I am not even sure that our explanations suffice to our students.

So my initial idea was to find a proof that would be of a more
systematic nature, using arguments that are more prone
to generalization. The principle of such a proof,
already hinted to in the book of~\cite{Wilson-2009},
is given by the Iwasawa criterion,
to which I now turn.

\section{The Iwasawa criterion for simplicity}

\cite{Iwasawa-1941} proposed a
proof of the simplicity
of the projective special linear group $\mathrm{PSL}(n,F)$
of a field $F$ of cardinality at least~$4$.
Before that, this theorem was limited
to the case of finite fields (Dickson) or fields of characteristic $\neq 2$
(van der Waerden).
From that proof, the following geometric criterion 
can be extrated.

\begin{theo}\label{theo.iwasawa}
Let a group~$G$ act on a set~$X$, and  assume that we are given,
for every $x\in X$, a subgroup~$T(x)$ of~$G$, such that 
the following properties hold:
\begin{itemize}
\item For every $x\in X$, the group~$T(x)$ is commutative;
\item For every $g\in G$ and every $x\in X$, one has $T(g\cdot x)
= g T(x)g^{-1}$;
\item The groups $T(x)$ generate~$G$.
\end{itemize}
If, moreover, the action of~$G$ on~$X$ is \emph{quasiprimitive},
then any normal subgroup~$N$ of~$G$ that acts
nontrivially on~$X$ contains the commutator subgroup~$D(G)$ of~$G$.
\end{theo}

An action of a group~$G$ on a set~$X$ is 
said to be \emph{quasiprimitive} if any normal subgroup
of~$G$ which acts nontrivially on~$X$ acts transitively.
This property may look obscure,
but it appears naturally in the framework of 
\emph{primitive}  actions,
a classic theme of 19th century group theory
which remained very important in
finite group theory but seems to have disappeared
from the algebra package we offer to undergraduate students. 
Let us define it
in terms of partitions of~$X$ (sets of nonempty disjoint subsets
of~$X$ whose union is~$X$):
\begin{defi}
A transitive action of a group~$G$ on a set~$X$ is \emph{primitive}
if there are exactly two partitions of~$X$ which are invariant under~$G$,
the coarse partition $\{X\}$ and the discrete partition 
consisting of all singletons.
\end{defi}

In particular, this definition implies that the
set~$X$ has at least two elements.

If $H$ is a normal subgroup of~$X$, then 
the partition of~$X$ in orbits of~$H$ is an invariant partition;
consequently, if the action of~$X$
is primitive, then either $H$ acts trivially, or it acts transitively:
for every $x,x'\in X$, there exists $h\in H$ such that $h\cdot x=x'$.
This shows that primitive actions are quasiprimitive.

Higher transitivity conditions give important examples of primitive actions.
\begin{lemm}
Let us assume that the action of~$G$ is 2-fold transitive: $X$ has
at least two elements and
for any two pairs $(x,y)$ and $(x',y')$ of distinct elements of~$X$,
there exists $g\in G$ such that $g\cdot x= x'$ and $g\cdot y=y'$.
Then  this action is primitive.
\end{lemm}
The proof is elementary: consider an
element~$B$ of a partition~$\Sigma$ of~$X$ 
which has at least two elements $x,y$ and let us show that $B=X$.
Let $z\in X$ be such that $z\neq x,y$. By 
the 2-fold transitivity condition applied to $(x,y)$ and $(x,z)$, 
there exists $g\in G$ such that $g\cdot x=x$ and $g\cdot y=z$. 
The set $g\cdot B$ belongs to~$\Sigma$ but has a common point with~$B$,
namely~$x$, 
so that $g\cdot B= B$. In particular, $z\in B$. This proves that $B=X$.

We just observed that members of a $G$-invariant partition
are subsets~$B$ of~$X$ such that either $g\cdot B\cap B=\emptyset$
or $g\cdot B=B$; 
in the traditional terminology  of group theory,
they are called \emph{blocks},
and blocks which are neither empty, nor singletons,
nor the full sets are called \emph{blocks of imprimitivity}.
Conversely,
if $B$ is a nonempty block and if the action
is transitive, then the set of all $g\cdot B$,
for $g\in G$, gives a $G$-invariant partition of~$X$.

As an example of a transitive, but not primitive action, 
one may consider the action of $\mathfrak S_4$
on the set of pairs of elements of $\{1,2,3,4\}$:
in this case,  there are nontrivial blocks,
such as $B=\{\{1,2\},\{3,4\}\}$. 
In fact, we will have to meet this example later, and some variants of it.

The terminology “primitive” comes from Galois, in the language
of equations: as explained by~\cite[p.~390]{Neumann-2006},
when the Galois group $G$ of an irreducible polynomial 
equation~$f(x)=0$
acts on its roots, there are $m$~blocks of size~$n$
if and only if there is an auxiliary equation of degree~$m$ 
the adjunction of one root of which allows~$f$ to be factored
as $f_1f_2$, where $f_1$ has degree~$n$.

The \texttt{Lean} definitions follow these descriptions, see listing~\ref{lst:blocks}, with a few adjustments to follow the general \texttt{mathlib} conventions.

\begin{lstlisting}[caption={Blocks, primitive actions},label=lst:blocks,captionpos=t,float=htb]
variables (G : Type*) {X : Type*} [has_smul G X]

/-- A trivial block is a subsingleton or ⊤ (it is not necessarily a block…)-/
def is_trivial_block (B : set X) := B.subsingleton ∨ B = ⊤

/-- A block is a set which is either fixed or moved to a disjoint subset -/
def is_block (B : set X) := (set.range (λ g : G, g • B)).pairwise_disjoint id

/-- An action is preprimitive if it is pretransitive and
the only blocks are the trivial ones -/
class is_preprimitive 
extends is_pretransitive G X : Prop :=
(has_trivial_blocks' : ∀ {B : set X}, (is_block G B) → is_trivial_block B)
\end{lstlisting}

First of all, definitions are always given under very minimal hypotheses,
one idea being that they could serve in more general contexts than
the ones that are generally considered, so as to avoid the need
for infinite variations of otherwise identical proofs.
Another principle to have definitions as general as possible
is that changing a definition later on requires
to adjust all theorems that refer to it, a painful and long task.
In our case, ``actions'' of a type \lstinline!G! on another type \lstinline!X!
just presumes a map \lstinline!G → X → X! embodied in the predicate \lstinline!has_smul G X!, and then denoted by the \mbox{symbol~\lstinline!•!},
not even requiring that \lstinline!G! has an inner multiplicative structure!
It is reminiscent of the “groups with operators” introduced
in the first chapter of \citep{zbMATH01194715} with a similar intention.

Then a “subset” \lstinline!B! of \lstinline!X! 
(something called \lstinline!set X!) is a block if and only if 
the sets \lstinline!g • B!, for \lstinline!g! in \lstinline!X!, are pairwise equal or disjoint.  The (possibly) obscure definition makes
use of \texttt{mathlib}'s general predicate \lstinline!set.pairwise_disjoint!.

Trivial blocks are detected by the predicate \lstinline!is_trivial_block!,
defined as either “subsingletons” (the empty set or a singleton,
with definition “$\forall x, y \in B, x =y$”)
or the full set \lstinline!⊤!.

Another \texttt{mathlib} idiosyncracy that appears in the definitions
above is the concept
of “pretransitive” actions, meaning “transitive but possibly empty”.
Again, the idea is to defer the non-emptiness hypotheses to
the statements that actually and explicitly need them. 
We thus define an action to be \lstinline!preprimitive! if 
it is pretransitive and if all blocks are trivial. 

In what follows,  it will be important to use 
the following equivalent characterization of primitive actions.
(Recall that the fixator~$G_x$ of an element~$x$ in~$X$
is the subgroup of~$G$ consisting of all $g\in G$ such that $g\cdot x=x$.)

\begin{lemm}\label{lemm.primitive-maximal}
The action of~$G$ on~$X$ is primitive if and only if
it is transitive and if for every $x\in X$, 
its fixator~$G_x$ is a maximal\footnote{%
Recall that a subgroup~$H$ of~$G$ is maximal if $H\neq G$
and if any subgroup $H'$ of~$G$ containing~$H$ is~$H$ or~$G$.} 
subgroup of~$G$.
\end{lemm}

More generally, 
one can show that 
for any $x\in X$, the mapping $H\mapsto H\cdot x$ 
induces an order preserving bijection
from the lattice of subgroups~$H$ of~$G$ such that
$G_x\subseteq H\subset G$ to  the lattice of blocks~$B$ in~$X$
that contain~$x$. We copied in listing~\ref{lst:ordereq}
the \texttt{Lean} definition of this order preserving bijection
(in fact, its inverse):
it takes the form of an “order equivalence” of types, 
as indicated by the symbols \lstinline!≃o!.
The first type, 
\lstinline!{ B : set X // a ∈ B ∧ is_block G B }!, is the type
of all \lstinline!B : set X! (basically, subsets of \lstinline!X!)
satisfying the properties that \lstinline!a ∈ B! and \lstinline!is_block G B!,
the latter type encoding that \lstinline!B! is a block for the action of \lstinline!G! on \lstinline!X! (which could be left implicit, 
because \lstinline!B! being of type \lstinline!set X!, this type is known).
\texttt{Lean} is capable to guess by itself that this
type inherits the ordering relation given by inclusion on \lstinline!set X!.
The second type, \lstinline!set.Ici (stabilizer G a)!,
designates, in the lattice \lstinline!subgroup G!, 
the subset of those subgroups containing \lstinline!stabilizer G a!.
This “order equivalence” consists of two
functions,
\lstinline!to_fun! and \lstinline!inv_fun!,
proofs (\lstinline!left_inv! and \lstinline!right_inv!)
that they are inverse of each other, and a proof (\lstinline!map_rel_iff!)
that they respect the order. 
Then comes the definition of
the function \lstinline!to_fun!, which maps such a \lstinline!B!,
together with the witnesses \lstinline!ha : a ∈ B!
and \lstinline!hB : is_block G B!, to \lstinline!stabilizer G B!,
accompanied with \lstinline!stabilizer_of_block hB ha!.
As one can guess, the former designates the stabilizer of \lstinline!B!
in \lstinline!G!, together with the additional information
that it contains \lstinline!stabilizer G a!. 
That information is provided by the function 
\lstinline!stabilizer_of_block : is_block G B → a ∈ B → stabilizer G a ≤ stabilizer G B!
whose code, of course, had been given earlier in the source.
The inverse function \lstinline!inv_fun! maps
\lstinline!H : subgroup G! together with \lstinline!hH : stablizer G a ≤ H!
to \lstinline!mul_action.orbit H a! which represents the
orbit of \lstinline!a! under the action of the subgroup \lstinline!H!,
together with the relevant proofs that this 
set contains \lstinline!a! and is a block. 
Then come three proofs, 
\lstinline!left_inv! and \lstinline!right_inv! 
stating that the two preceding functions are inverse of each other, 
while \lstinline!map_rel_iff! states that they respect the order relation. 
In the listing showed there,
we replaced these three proofs by \lstinline!…!; the codes of the first
two ones are 2-line long, that of the third one is 17-line long.

\begin{lstlisting}[caption={Order equivalence between blocks
containing a point and subgroups containing its stabilizer},float=hbt,label=lst:ordereq,captionpos=t]
variables {G: Type*} [group G] {X : Type*} [mul_action G X]
/-- Order equivalence between blocks in X containing a point a
 and subgroups of G containing the stabilizer of a 
 (Wielandt, Finite Permutation Groups, th. 7.5)-/
def stabilizer_block_equiv [htGX : is_pretransitive G X] (a : X) :
  { B : set X // a ∈ B ∧ is_block G B } ≃o set.Ici (stabilizer G a) := {
to_fun := λ ⟨B, ha, hB⟩, ⟨stabilizer G B, stabilizer_of_block hB ha⟩,
inv_fun := λ ⟨H, hH⟩, ⟨mul_action.orbit H a,
  mul_action.mem_orbit_self a, is_block_of_suborbit hH⟩,
left_inv := …, 
right_inv := …, 
map_rel_iff' := …,
end }
\end{lstlisting}

\subsection{}
We end this section with a proof of the Iwasawa criterion
(theorem~\ref{theo.iwasawa}). 

Fix a point $a\in X$. 
We first prove that the subgroup $\langle N, T(a)\rangle$ generated by~$N$
and~$T(a)$ is equal to~$G$.  By assumption, $N$
acts transitively on~$X$. Since $N$ is normal,
the hypothesis that the action is quasiprimitive implies that 
for every $b\in X$,
there exists $n\in N$ such that $n\cdot a=b$.
Since $n T(a) n^{-1} = T(b)$, this implies
that $\langle N,T(a)\rangle$ contains~$T(b)$.
Since $b$ is arbitrary, the subgroup $\langle N,T(a)\rangle$
contains the subgroup generated by all $T(x)$, for $x\in X$,
which, by assumption, is~$G$.

The subgroup~$N$ is normal; the desired conclusion that it is contains
the derived subgroup of~$G$ is equivalent to the
commutativity of the quotient $G/N$. Since $\langle N,T(a)\rangle = G$,
the composition $T(a) \to G \to G/N$ is surjective;
since $T(a)$ is commutative, we conclude that $G/N$ is commutative,
as we wished to.

\section{Normal subgroups of symmetric and alternating groups}

In this section, we consider an integer~$n$;
we generally assume that $n\geq 5$.
 
The symmetric group~$\mathfrak S_n$
acts not only on the set~$X=\{1,\dots,n\}$,
but also on the sets~$X^{[k]}$ of $k$-element subsets of~$X$,
for any integer~$k$ such that $0\leq k\leq n$.
The action is trivial if $k=0$ or $k = n$, because 
then $X^{[k]}$ is reduced to a single element,
but it is faithful otherwise: any element $g\neq e$
acts nontrivially.
The following proposition asserts that this action is moreover primitive,
unless $n=2k$.

\begin{prop}\label{prop.maximal}
Let $k$ and~$n$ be integers such that $0<k<n-k<n$.
If $4\leq n$, then the actions of~$\mathfrak A_n$  and~$\mathfrak S_n$
on~$X^{[k]}$ are primitive.
\end{prop}

Given this primitivity result, the approach of Iwasawa allows us
to understand the normal subgroups of the symmetric and alternating groups.
We will only need to use the cases $k=2$, $k=3$ and~$k=4$.

\subsection{}\label{ss.prim-2}
Let us first consider the case $k=2$.
For any 2-element subset $x=\{a,b\}$ of~$X$,
let us consider the subgroup~$T(x)$ generated by the transposition~$(a\,b)$:
it is commutative of order~$2$;
the relation $(g\cdot a\, g\cdot b)=g (a\,b)g^{-1}$ implies that
these subgroups satisfy the relation $T(g\cdot x)=g T(x) g^{-1}$;
and since $\mathfrak S_n$ is generated by all transpositions,
they generate the symmetric group.
Consequently, Iwasawa's criterion implies that
if this action  is primitive, then any normal subgroup~$N$
of~$\mathfrak S_n$ such that $N\neq\{e\}$ 
contains~$D(\mathfrak S_n)$, which as we have
seen, is equal to~$\mathfrak A_n$.
Since $\mathfrak S_n/\mathfrak A_n$ has order~$2$,
the only subgroups of~$\mathfrak S_n$ that contain~$\mathfrak A_n$
are $\mathfrak A_n$ and $\mathfrak S_n$. 

What about the primitivity assumption?
Note that the action of~$\mathfrak S_n$ on~$X^{[2]}$
is not 2-fold transitive, because one cannot map $\{1,2\}$ and $\{1,3\}$
to the sets $\{1,2\}$ and $\{3,4\}$. 
Let us observe that it is nevertherless primitive; 
here, we will use that $2<n-2$, that is, $n > 4$.
\cite[\S2.5.1]{Wilson-2009} shows that the fixator of any element of~$X^{[2]}$
is a maximal subgroup, and we will discuss this
in greater generality in the next section,
but let me tell right now the following proof
as explained to me by G.~Chenevier.

Let $B$ be an imprimitivity block of~$X^{[2]}$, and let $\{a,b\}$
be a pair in~$B$.

First assume that $B$ contains another pair of the form~$\{a,c\}$.
Consider $g\in G$ such that $g\cdot a=c$
and $g\cdot b=a$; then $B$ and $g\cdot B$
share the element $\{a,c\}$, so that $g\cdot B=B$;
consequently, $B$ contains the pair~$\{g\cdot a, g\cdot c\}=\{c, g\cdot c\}$,
hence all pairs of the form~$\{c,d\}$.
Redoing the argument from $\{a,c\}$ and $\{c,d\}$
we deduce that $B$ contains any pair, hence $B=X^{[2]}$.

Assume then that $B$ contains a pair $\{c,d\}$ 
which is disjoint from~$\{a,b\}$. Since 
$n\geq 5$, we may consider  a fifth element~$e$ in~$X$;
let us prove that $\{c,e\}\in B$.
Indeed, there exists $g\in\mathfrak S_n$
which maps $a$ to~$a$, $b$ to~$b$, $c$ to~$c$ and~$d$ to~$e$,
hence $\{a,b\}$ to itself, and $\{c,d\}$ to~$\{c,e\}$;
then $B$ and~$g\cdot B$ have $\{a,b\}$ in common, so that $g\cdot B=B$
and $\{c,e\}\in B$. In particular, $B$
contains two pairs $\{c,d\}$ and $\{c,e\}$ whose supports
are not disjoint and the first part of the argument implies that $B=X^{[2]}$.

We thus obtain the following result (also a consequence
of theorem~\ref{theo.an-simple}).
\begin{prop}
For $n\geq 5$, the normal subgroups of~$\mathfrak S_n$
are $\{e\}$, $\mathfrak A_n$ and $\mathfrak S_n$.
\end{prop}

\subsection{}
We now pass to $k=3$.
For any 3-element set $x=\{a,b,c\}$ in~$X$, 
we consider the alternating group~$T(x)$ of these three elements,
viewed as a subgroup of $\mathfrak A_n$; it is
the subgroup generated by the 3-cycle $(a\,b\,c)$.
As above, the relations $T(g\cdot x)=gT(x)g^{-1}$ hold,
and these subgroups generate the alternating group.
Assuming that the action of~$\mathfrak A_n$
on~$X^{[3]}$ is primitive, we deduce from Iwasawa's criterion
that any normal subgroup of~$\mathfrak A_n$
either is trivial, or contains~$D(\mathfrak A_n)$;
in other words, $\mathfrak A_n$ is simple.

We shall see in the next section that the primitivity
condition holds for $n\neq 6$ (and why it does not for $n=6$),
so the case $n=6$ requires another argument.

\subsection{}
For this, let us consider $k=4$.
For any 4-element set $x=\{a,b,c,d\}$ in~$X$,
let us consider Klein's Vierergruppe~$V(x)$
in the alternating group of these four elements,
viewed as a subgroup of $\mathfrak A_n$.
It is commutative of order~4, and consists
of the identity and of the three “double transpositions”
$(a\,b)(c\,d)$, $(a\,c)(b\,d)$ and $(a\,d)(b\,c)$.
This is already an intrinsic definition of~$V(x)$
(permutations with support in~$x$ whose cycle type is either
empty or $(2,2)$); it can also be defined as the derived subgroup of
the alternating group on these four elements.
Consequently, the relations $V(g\cdot x)=g V(x)g^{-1}$ hold.
Let us show that these subgroups~$V(x)$ generate~$\mathfrak A_n$;
the argument will use that $n\geq 5$.
We start from the remark that $\mathfrak A_n$ consists
of permutations which are products of an even number
of transpositions. If two successive transpositions
in such a product have disjoint supports, they belong to some~$V(x)$.
Otherwise, if their supports share an element~$a$,
say $(a\,b)(a\,c)$, then 
using that $n\geq 5$,
we can insert a cancelling product $(d\,e)(d\,e)$,
so that $(a\,b)(d\,e)$ and $(d\,e)(a\,c)$ belong to
subgroups of the form~$V(x)$.

Applying Iwasawa's criterion,
this construction shows that the alternating group~$\mathfrak A_n$
is simple  as soon as 
the action of~$\mathfrak A_n$ on~$X^{[4]}$ is primitive.

\subsection{}
We note that a variant of these arguments leads 
to a reasonably simple proof
of the simplicity of~$\mathfrak A_5$. Indeed, by taking
complements, the action of~$\mathfrak A_n$ on~$X^{[k]}$
is isomorphic to the action on~$X^{[n-k]}$. When $n=5$,
the case $k=4$ is reduced to the case $k=1$ and it suffices
to prove that the action of~$\mathfrak A_5$ on~$X$ is primitive.
To that aim, it suffices to observe that the action of~$\mathfrak A_5$
is 2-transitive. (In fact,  it is even 3-transitive.) 

\section{Primitivity and maximal subgroups}

To conclude the proof of theorem~\ref{theo.an-simple},
it remains to explain the proof of proposition~\ref{prop.maximal}.
The fixator of the element $\{1,\dots,k\}$ of~$X^{[k]}$ 
is the intersection
of~$\mathfrak A_n $ with the subgroup $\mathfrak S_k\times\mathfrak S_{n-k}$
associated with the partition of~$\{1,\dots,n\}$
in $\{1,\dots,k\}$ and $\{k+1,\dots,n\}$.
Since the action of~$\mathfrak A_n$ on~$X^{[k]}$ is transitive,
lemma~\ref{lemm.primitive-maximal} reduces us  to prove
that this subgroup is a maximal subgroup of~$\mathfrak A_n$.

In this way, we note that the hypothesis $n\neq 2k$ is really
necessary for this proposition: 
the subgroup $\mathfrak S_n\times\mathfrak S_n$
of $\mathfrak S_{2n}$ is not maximal, it is a subgroup of index~2
of the stabilizer of the partition $\{\{1,\dots,n\},\{n+1,\dots,2n\}\}$,
a group also described as the wreath product $\mathfrak S_n\wr (\Z/2\Z)$. 

\subsection{}
In \S\ref{ss.prim-2},
we saw an elementary proof of proposition~\ref{prop.maximal}
for $k=2$ and $n\geq 5$,
and it seems likely that an elementary proof exists for any~$k$.
R.~Rouquier gave me one that works for $k=3$ and $n\geq 7$.
However, I want to describe another approach, explained to
me by M.~Liebeck, that I think emphasizes the status of 
that proposition within finite group theory.

One of the first treatises on group theory is that of \cite{Jordan-1870}.
Then groups were “permutation groups”, permuting letters,
(henceforth algebraic expressions on these letters),
or, since the connection with the Galois theory of equations was explicit,
roots of a polynomial equation.

It had been observed that the symmetric group on $n$~letters
is $n$-fold transitive --- almost by definition, 
given two systems of distinct elements in $\{1,\dots,n\}$,
$x_1,\dots,x_n$ and $y_1,\dots,y_n$, say, 
there is a permutation~$g$ such that $g\cdot x_i=y_i$ for all~$i$,
and $g$ is even unique.

Only slightly less obvious was the fact that the alternating
group on $n$~letters is $(n-2)$-fold transitive:
given distinct systems $x_1,\dots,x_{n-2}$ and $y_1,\dots,y_{n-2}$,
there are exactly two permutations~$g,g'$ such that $g\cdot x_i=g'\cdot x_i=y_i$ for all~$i$, and $g'g^{-1}$ is the permutation that exchanges
the two elements of~$\{1,\dots,n\}$
not in $\{y_1,\dots,y_{n-2}\}$; in particular, one of them is even
and the other is odd.

It had also been observed that beyond these two cases, 
a permutation subgroup on $n$~letters 
has to act much less transitively and 
19th century mathematicians 
proved many theorems that aimed at quantifying this limit.
For example, Mathieu had proved that unless it contains
the alternating group, a subgroup of~$\mathfrak S_n$
isn't $n/2$-fold transitive, 
while \citet{Jordan-1872a} proved that it isn't $m$-fold transitive 
if $n-m$ is a prime number $>2$.

As explained in~\cite{Cameron-1981},
once the classification of simple finite groups had been achieved,
it could be checked on the list that a 6-fold transitive
subgroup of $\mathfrak S_n$
must be symmetric or alternating.

Parallel to the classification is the understanding
of all \emph{maximal} subgroups of a given finite simple group.
In the case of the alternating group,
an explicit list has been provided independently by M.~O'Nan and L.~Scott.
As remarked by~\cite{Cameron-1981},
this question is closely related to the description of all subgroups
of the symmetric group~$\mathfrak S_n$ which 
act primitively on~$\{1,\dots,n\}$.

This classification theorem takes the given form: Let $G$ be a strict
subgroup of $\mathfrak A_n$ or~$\mathfrak S_n$; then
$G$ is conjugate to a subgroup of one of six types
of which the first three take the form:
\begin{enumerate}\def\theenumi{\alph{enumi}}\def\labelenumi{(\theenumi)}
\item A product $\mathfrak S_m \times\mathfrak S_{n-m}$, where $0<m<n$
--- the \emph{intransitive case}.
\item The “wreath product” $\mathfrak S_m \wr \mathfrak S_p$, where $n=pm$,
namely the subgroup generated by 
the product of $p$~symmetric groups acting on $p$~disjoint sets 
of $m$~letters (isomorphic to $\mathfrak S_m \times \dots\times \mathfrak S_m)$,
and a permutation that permutes cyclically these $p$~sets
--- the \emph{imprimitive case}; 
\item An \emph{affine group} of an $\F_p$-vector space of dimension~$d$,
where $n=p^d$ is the power of a prime number.
\end{enumerate}
It applies in particular to maximal subgroups, and
\cite{LiebeckPraegerSaxl-1987} established the converse assertion,
deciding which of the groups of this list are maximal.
That case~(a) is maximal when $m\neq n-m$ is exactly the
statement of proposition~\ref{prop.maximal}.
However, when $n=2m$, case~(a) is not maximal but case~(b) gives
the corresponding maximal case. For $n=4$, for example,
the subgroup given by~(b) has order~8,
hence is a 2-sylow subgroup of~$\mathfrak S_4$,
while the subgroup $\mathfrak S_2\times\mathfrak S_2$ has order~4.

Cases of the form~(c) were of particular interest to Galois,
who proved that they appear for the Galois groups
of irreducible equations of prime degree which are solvable by radicals.
In other words, solvable and transitive subgroups of~$\mathfrak S_p$
can be viewed, up to conjugacy, as a group of permutations 
of the form $x\mapsto ax+b$ on~$\F_p$, for $a\in\F_p^\times $ and $b\in\F_p$.
Since the identity is the only permutation of that form that fixes
two elements, Galois obtains that an irreducible equation of prime
degree is solvable by radicals if and only if 
any of its roots can be expressed rationally by any two of them.

Galois also defined primitive algebraic equations 
which correspond exactly to the case where the Galois group
acts primitively on their roots. In the solvable case,
he proved that the degree has to be a power~$p^n$ of a prime number~$p$
and, up to an enumeration of the vector space~$\F_{p}^n$,
the Galois group~$G$ is a subgroup of the group
of permutations of the form $x\mapsto Ax+b$, for $A\in\mathrm{GL}(n,\F_p)$
and $b\in\F_p^n$, that contains all translations $x\mapsto x+b$.
Moreover, the subgroup~$G_0$ of $\mathrm{GL}(n,\F_p)$
consisting of all elements of~$G$ of the form $x\mapsto Ax$
has no nontrivial invariant subspace.
The interested reader shall find more details on this fascinating
story in chapter~14 of~\cite{Cox-2012}.

\subsection{}
But let us go back to the promised proof of proposition~\ref{prop.maximal}.
Let $G$ be a subgroup of $\mathfrak A_n$
strictly containing $(\mathfrak S_k\times \mathfrak S_{n-k})\cap \mathfrak A_n$,
where $0<k<n$ and $n\neq 2k$.
We need to prove that $G$ coincides with~$\mathfrak A_n$.
By symmetry, we may assume that $k<n-k$.
The case $k=1$ is easy. Indeed, the action of $\mathfrak A_n$
on~$\{1,\dots,n\}$ is $(n-2)$-fold transitive, hence
it is 2-fold transitive, because $n\geq 4$, hence it is primitive.
We now assume that $2\leq k$; then $n\geq 5$.

A theorem of~\cite[Note C to \S398, page 664]{Jordan-1870} asserts
that a primitive subgroup of~$\mathfrak S_n$ that
contains a cycle of prime order~$p$ is at least $(n-p+1)$-fold transitive.
When $p=2$, we get that this subgroup is $(n-1)$-fold transitive,
hence it has to be the whole~$\mathfrak S_n$, while
when $p=3$, it is $(n-2)$-fold transitive, and it is not too
difficult to deduce that it contains~$\mathfrak A_n$.
Since $1\leq k<n-k<n$ and $n\geq 5$, we have $n-k\geq 3$
and our subgroup~$G$ contains a 3-cycle. To conclude, it remains
to establish that it acts primitively on~$\{1,\dots,n\}$.

One first proves that $G$ acts transitively on~$\{1,\dots,n\}$.
In fact, $G$ contains the subgroups $\mathfrak S_k$ and~$\mathfrak S_{n-k}$;
in particular, it acts transitively on the elements of each subset
$\{1,\dots,k\}$ and $\{k+1,\dots,n\}$, hence it has at most two orbits.
But since it strictly
contains $(\mathfrak S_k\times \mathfrak S_{n-k})\cap \mathfrak A_n$,
it cannot leave $\{1,\dots,k\}$ and $\{k+1,\dots,n\}$ invariant. 

Arguing as for transitivity, $G$ acts $k$-fold transitively
on $\{1,\dots,k\}$ and $(n-k)$-fold transitively on~$\{k+1,\dots,n\}$;
since $2\leq k<n-k$, it acts in particular 2-fold transitively,
hence primitively, on both of these sets.

We consider imprimitivity blocks~$B$ for the action of~$G$,
assuming that they have at least two elements and are distinct
from~$\{1,\dots,n\}$.

First observe that $B$ cannot contain~$\{k+1,\dots,n\}$, because its 
translates $g\cdot B$, for $g\in B$ such that $g\cdot B\neq B$, 
would have to be contained in~$\{1,\dots,k\}$, 
which is impossible since $k<n-k$. In particular,
$B$ meets~$\{k+1,\dots,n\}$ in at most one element.
If it is disjoint from $\{k+1,\dots,n\}$, it is contained
in~$\{1,\dots,k\}$.
Since $G$ acts primitively on~$\{1,\dots,k\}$,
one then has $B=\{1,\dots,k\}$.
Consider an element~$g$ of~$G$ which does not stabilize $\{1,\dots,k\}$.
Then $g\cdot B$ is a block distinct from~$B$, hence disjoint,
so that $g\cdot B$ is a block contained in~$\{k+1,\dots,n\}$.
By primitivity, $g\cdot B=\{k+1,\dots,n\}$, contradicting
the beginning of the proof.

In particular, there are elements $a\in\{1,\dots,k\}\cap B$
and $b\in\{k+1,\dots,n\}\cap B$. To conclude the proof
by a contradiction,
it suffices to establish that $B$ contains \mbox{$\{k+1,\dots,n\}$}.
So let $c\in\{k+1,\dots,n\}$ and 
consider an element $g\in G$ that fixes $\{1,\dots,k\}$ 
such that $g\cdot b= c$.
Then $g\cdot B$ and $B$ both contain~$a$, hence $g\cdot B=B$,
hence $c\in B$, as was to be shown.

\section{Intermezzo: conjugacy classes in symmetric groups}

At the end, the proof of the simplicity theorem 
of the alternating group $\mathfrak A_6$ required 
a discussion of the Klein subgroup of $\mathfrak A_4$. 
When we discuss this group between colleagues, possibly in class,
the proof fact that it is indeed a subgroup usually boils down to
a mere: “one checks that…”. 

Of course, such an argument is not sufficient for the computer,
and I spent some time trying to imagine how we should prove such
facts in the computer. 
The proof I resorted to happened to be fun, 
nevertheless slightly sophisticated.

Let $X=\{a,b,c,d\}$ be a set with four elements, 
and let $V$ be the subset of $\mathfrak S_X$ consisting 
of the identity and of all double transpositions.
In order to prove that $V$ is a subgroup of $\mathfrak S_X$,
I prove that $V$ is the only 2-sylow subgroup of $\mathfrak A_X$.
The proof runs as follows, in which we  consider an arbitrary
 2-sylow subgroup~$S$ of $\mathfrak A_X$.
\begin{itemize}
\item Since $\mathfrak S_4$ has cardinality~$4!=24$,
the alternating group $\mathfrak A_4$ has cardinality~12, 
and $S$ has cardinality~4.
\item The order of any element~$g$ of $S$ divides~4;
since its entries thus divide~4,
the cycle type of~$g$ belongs to $()$, $(2)$, $(2,2)$ or $(4)$.
Since the second and last cases give odd permutations, 
we have $g\in V$;
\item Now the number of permutation of a given cycle type
in a symmetric group can be computed explicitly, more on this below, 
and the computation shows that $V$ has 4~elements; 
\item Since $S$ and $V$ both have 4~elements, and $S\subseteq V$,
this proves $V=S$, as claimed.
\end{itemize}

The computation of the number of permutations of a given cycle type in
the symmetric group~$\mathfrak S_X$
is by itself a classic and important result in combinatorics of finite 
permutation groups. We return to the general case of a finite set~$X$,
let $n$ be its cardinality, 
and consider a partition~$\pi$ of the integer~$n$;
let us write $m_i$ for the number of parts equal to~$i$. 
A permutation of cycle type~$\pi$
takes the form $(a_1)(a_2)\dots(a_{m_1})(b_1\, b'_1)(b_2\,b'_2)\dots (b_{m_2}\,b'_{m_2})\dots $: $m_1$~cycles of length~1, $m_2$~cycles of length~2, etc.
In order to compute the number of permutations
of cycle type~$\pi$, we just have to fill the letters with distinct
elements of~$X$, which apparently makes for~$n!$ permutations.
However, for each cycle of length~$i$,
only the cyclic ordering of the elements matters,
so we have to divide the result by $\prod i^{m_i}$.
Moreover, the order in which we write the $m_i$~cycles of given length~$i$
does not matter, so the result needs to be further divided
by $\prod m_i!$.
Finally, the number of permutations of cycle type~$\pi$ is 
${n!}/{\prod i^{m_i} \prod m_i!}$.

There is however a more conceptual, and more precise, way 
to prove this formula. 
Fix any permutation~$g$ which has cycle type~$\pi$.
Since the number we wish to compute is the cardinality
of the orbit of~$g$ under the conjugation action, it suffices
to prove that the cardinality of the centralizer~$Z_g$ of~$g$
is equal to $\prod i^{m_i}\prod m_i!$.

If $h\in Z_g$, then $hgh^{-1}=g$, so that the cycles of~$hgh^{-1}$
are those of~$g$. In other words, $Z_g$ acts by conjugation
on the set of cycles of~$g$, respecting their lengths.
This gives a group morphism $\phi\colon Z_g \to \prod_i \mathfrak S_{m_i}$.

This morphism is surjective.  In fact, one can even show that $\phi$ has a section. Indeed, fix, for each cycle~$c$ of~$g$, an element $a_c$ in~$c$;
then, for any permutation~$\sigma$ of the set of cycles of~$g$
which preserves their lengths, there is a unique element~$h_\sigma$ of~$Z_g$
such that $h_\sigma(a_c)= a_{\sigma(c)}$ for all~$c$, and the map $\sigma\mapsto h_\sigma$ is a group morphism.

Now, the kernel of~$\phi$ is the subgroup of all elements~$h\in Z_g$
such that $hch^{-1}=c$ for all cycles~$c$ of~$g$.
Necessarily, $h$ stabilizes the support of each such~$c$,
so it maps~$a_c$ to some power iterate of~$a_c$ under~$g$;
fix~$k_c\in\Z$ (modulo the cardinality~$n_c$ of the support of~$c$) 
such that $h(a_c)=g^{k_c}(a_c)=c^{k_c}(a_c)$; using
the fact that~$c$ is a cycle, it follows that~$h$ acts
like~$c^{k_c}$  on the support of~$c$. Finally, we see that
$h$ is the product of these powers $c^{k_c}$. 
In other words, $\ker(\phi)$ is a product of cyclic groups,
$\prod_c (\Z/k_c\Z)$, which we rewrite as 
$ \prod_i (\Z/i\Z)^{m_i}$, since $m_i$ is
the number of cycles~$c$ such that $n_c=i$.
In particular, the order of~$\ker(\phi)$ is equal
to $\prod i^{m_i}$.

Finally, $\Card(Z_g) = \Card(\im(\phi))\Card(\ker(\phi))
= \prod_i i^{m_i} \prod_i m_i!$, as was to be shown.

\section{Simplicity of classical groups}

\subsection{}
The simplicity criterion is not explicitly stated 
by~\cite{Iwasawa-1941}, but it is directly proved and applied   
in the case of the projective special linear group~$\mathrm{PSL}(n,F)$ 
acting on the projective space $\mathbf P_{n-1}(F)$
of lines in~$F^n$.  Unless $F$ has~$2$ or~$3$ elements,
a linear algebra argument shows that this action is 2-fold transitive,
hence primitive.

For every line~$\ell\in\mathbf P_{n-1}(F)$, consider
the subgroup $T(\ell)$ of transvections with respect 
to~$\ell$, namely the elements $g\in\mathrm{SL}(n,F)$
such that the range of $g-\id$ is contained in~$\ell$.
Using that $\mathrm{SL}(n,F)$ is generated by transvections,
we see that they give rise to a datum as in Iwasawa's criterion.
Consequently, any normal subgroup of~$\mathrm{SL}(n,F)$
which acts nontrivially on~$\P_{n-1}(F)$ contains
the commutator subgroup of~$\mathrm{SL}(n,F)$,
which is known to be $\mathrm{SL}(n,F)$ itself.
Finally, the only elements of~$\mathrm{SL}(n,F)$
which  act trivially on~$\P_{n-1}(F)$
are the homotheties, and they form the center of~$\mathrm{SL}(n,F)$,
a finite subgroup isomorphic to the set of $n$th roots of unity in~$F$.
As a consequence, the quotient $\mathrm{PSL}(n,F)
=\mathrm{SL}(n,F)/Z(\mathrm{SL}(n,F))$
is a simple group.

\subsection{}
This reasoning can also be applied for other cases of geometric groups.
In his paper, Iwasawa himself indicates that 
the same method works
for the symplectic group $\mathrm{PSp}(2n,F)$
(“complex projective groups” in the earlier terminology) acting
on the projective space $\mathbf P_{2n-1}(F)$.
Iwasawa does not explicitly consider the notion 
of a primitive action in his paper: his arguments 
are only spelt out for a 2-fold transitive action.
However, he mentions in a footnote that 
while the action of the symplectic group on $\mathbf P_{2n-1}(F)$
is not 2-fold transitive, it is quasiprimitive,
and this suffices for his proof.  On the other
hand, \cite{King-1981a} established that the stabilizers
of this action are maximal subgroups, so that this action is even primitive.

In fact, it seems that the simplicity of the appropriate
groups of geometric transformations can all 
be established in this way.

I find it remarkable how much this method,
that relates the simplicity of a group with the structure
of its maximal subgroups, is absolutely in par with the point
of view of Jordan and early group theorists!

\section{Remarks regarding the formalized proof and the formalization process}

As recalled in the introduction,
implementation of mathematical proofs in computers is not a very recent
activity, but the \texttt{Lean/mathlib} movement
puts us at a crossroad in so that it makes
an indefinitely-extending library of formalized proofs
conceivable. 
Built on the experiment described in this note,
I would like to risk myself to enumerating some remarks about
this prospect, the hopes and fears it may rise.

Formalization of mathematical proofs has many goals.

Some of its proponents raise the idea that it will make us 
truly certain of the validity of the new mathematical theorems we prove.
The idea here is that the traditional peer review seems to reach its limits,
both for mathematical and sociological reasons.

Some papers are simply so complicated that nobody can reasonably
claim to have checked their validity with absolute certainty.
This was the case for Hales's proof of the Kepler conjecture,
before he, leading a team of 21 mathematicians, formalized that proof.
In some sense, this is still the case 
for the classification of finite simple groups, whose size
and technicality makes it inacessible to most of
the mathematical community.

On the other hand, most research papers are of an apparently smaller size, 
but the sociology of the fields evolved. The increasing importance
of research grants for the funding of research, if not for obtaining
permanent academic positions, led us to a stage where  the collectivity
wants their papers published more quickly that it can assert their validity,
if not just read.
As a consequence, papers are reviewed too quickly, 
their publication is conditioned to preliminary opinions, 
leading to all imaginable biases, 
and new journals are created to host this ever growing mathematical
litterature. 

If we could check ourselves the validity of our proofs within
formalization software,  and deliver it at the same time
we submit a paper, 
it is likely that this paper could have been written in a different way:
not just to quickly convince a referee that the proofs are true,
but spending more time than we presently do
to explain the statements, 
their interest, their context,
the path that led to their proof,
as well as aiming at a possibly larger audience.

For this to happen, we need a \emph{huge} archive 
of mathematical proofs written in a common language,
with common definitions.  
The experience of the Bourbaki books suggests that something
is possible, but it also reminds that not all 
mathematicians will be willing to comply to the mathematical writing
style of other.

If the style of Bourbaki has been sometimes defined as too abstract, 
it is nothing in comparison with that of \texttt{mathlib}. Indeed,
in order to avoid repeating proofs, the authors of that library
make a permanent effort to put its definitions and statements 
in a (natural, but possibly frightening) generality. Linear algebra
starts with a discussion of semimodules over monoids, 
so that the relevant part applies to more exotic contexts,
such as $(\max,+)$-algebras. (Bourbaki 
made a similar step when they defined “groups with operators”,
but their notion does not seem to be commonly used.)
In complex analysis,
the characterization of analytic functions as functions
which are differentiable in the complex sense is proved
using the Kurzweil--Henstock integral, because that allows us
to avoid any Lebesgue integrability assumption on the derivative.

One of the difficulties by working with many simultaneous group 
actions
is that \emph{type theory}, the inner language of \texttt{Lean},
does not allow the many abuses of language that we do while
doing mathematics, without even thinking about it. 
Take, for example, a group $G$ acting on a set~$X$, 
a subset $A$ of~$X$, and a point $a\in X\setminus A$. Then we can consider
the fixator~$G_A$ of~$A$ in~$X$, and its action on $X\setminus A$,
then the fixator~$G_{A,a}$ of~$a$ in~$G_A$, and its
action on $X\setminus(A\cup\{a\})$,
which --- obviously --- coincides with the 
action of the fixator of $A\cup\{a\}$ on its complement.
However, these actions look sufficiently different to \texttt{Lean},
syntaxically, and it is not able to identify automatically.
The suggestion I received 
on the \emph{Zulip} discussion blackboard
was that I should not even try to identify them, 
but that it was sufficient to relate them  through \emph{equivariant maps}.
If groups $G$ and~$H$ act on~$X$ and~$Y$,
and $\phi\colon G \to H$ is a morphism  of groups,
then a $\phi$\nobreakdash-equivariant map from $X\to Y$ is just a map $f\colon X\to Y$ 
such that $f(g\cdot x)=\phi(g)\cdot f(x)$ for all $g\in G$ and $x\in X$.
Then several basic results allow us to transfer primitivity
or transitivity properties from the action of~$G$ on~$X$
to the action of~$H$ on~$Y$, or vice versa. This is an example
of an elementary definition, with basic companion results, 
that we probably wouldn't dare introducing explicitly in a
standard mathematical discussion --- probably too trivial for specialists,
but already too obscure for beginners.
Learning to appreciate the relevance of introducing such abstract concepts
takes some time, requires a community of knowledgeable mathematicians,
as well as the will to follow their point of view.

On the other hand, as the myth of the Babel tower
should remind us, leaning towards some ever-expanding generality
comes at a high risk, that the full edifice collapses.
My own experiment has been agreeable enough to me
to sincerely wish that the community skillfully avoids that risk.

\bibliography{jordan}

\begin{thebibliography}{20}
\ProvideTextCommand{\guillemotleft}{OT1}{%
  \leavevmode\raise .27ex\hbox{$\scriptscriptstyle\ll$}}
\ProvideTextCommand{\guillemotright}{OT1}{%
  \leavevmode\raise .27ex\hbox{$\scriptscriptstyle\gg$}}
\newcommand{\enquote}[1]{\og #1\fg}
\providecommand{\og}{``}\providecommand{\fg}{''}
\expandafter\ifx\csname natexlab\endcsname\relax\def\natexlab#1{#1}\fi
\expandafter\ifx\csname url\endcsname\relax
  \def\url#1{\texttt{#1}}\fi
\expandafter\ifx\csname urlprefix\endcsname\relax\def\urlprefix{URL }\fi
\newcommand{\Capitalize}[1]{\uppercase{#1}}
\newcommand{\capitalize}[1]{\expandafter\Capitalize#1}
\providecommand{\eprint}[2][]{\url{#2}}

\bibitem[{Bourbaki(1998)}]{zbMATH01194715}
N.~\textsc{Bourbaki} (1998), \emph{Elements of Mathematics. {{Algebra I}}.
  {{Chapters}} 1–3. {{Transl}}. from the {{French}}. {{Softcover}} Edition of
  the 2nd Printing 1989.}, {Berlin: Springer}, softcover edition of the 2nd
  printing 1989 \bbledition{}.

\bibitem[{Cameron(1981)}]{Cameron-1981}
P.~J. \textsc{Cameron} (1981), \enquote{Finite {{Permutation Groups}} and
  {{Finite Simple Groups}}}. \emph{Bulletin of the London Mathematical
  Society}, \textbf{13}~(1), \bblpp{} 1--22.

\bibitem[{Cox(2012)}]{Cox-2012}
D.~A. \textsc{Cox} (2012), \emph{Galois Theory}, Pure and Applied Mathematics,
  {John Wiley \& Sons}, {Hoboken, N.J}, 2nd ed \bbledition{}.

\bibitem[{Doxiadis \& Papadimitriou(2009)}]{DoxiadisPapadimitriou-2009}
A.~\textsc{Doxiadis} \& C.~H. \textsc{Papadimitriou} (2009),
  \emph{Logicomix—{{An}} Epic Search for Truth}, {Bloomsbury Press, New
  York}.

\bibitem[{Gonthier(2008)}]{Gonthier-2008}
G.~\textsc{Gonthier} (2008), \enquote{Formal proof—the four-color theorem}.
  \emph{Notices of the American Mathematical Society}, \textbf{55}~(11),
  \bblpp{} 1382--1393.

\bibitem[{Gonthier \emph{\bbletal{}}(2013)Gonthier, Asperti, Avigad \&
  {al.}}]{GonthierAspertiAvigadEtAl-2013}
G.~\textsc{Gonthier}, A.~\textsc{Asperti}, J.~\textsc{Avigad} \&
  e.~\textsc{{al.}} (2013), \enquote{A machine-checked proof of the odd order
  theorem}. \emph{Interactive Theorem Proving}, Lecture Notes in Comput.
  {{Sci}}.~\textbf{7998}, \bblpp{} 163--179, {Springer, Heidelberg}.

\bibitem[{Hales \emph{\bbletal{}}(2017)Hales, Adams, Bauer, Dang, Harrison,
  Hoang, Kaliszyk, Magron, McLaughlin, Nguyen, Nguyen, Nipkow, Obua, Pleso,
  Rute, Solovyev, Ta, Tran, Trieu, Urban, Vu \&
  Zumkeller}]{HalesAdamsBauerEtAl-2017}
T.~\textsc{Hales}, M.~\textsc{Adams}, G.~\textsc{Bauer}, T.~D. \textsc{Dang},
  J.~\textsc{Harrison}, L.~T. \textsc{Hoang}, C.~\textsc{Kaliszyk},
  V.~\textsc{Magron}, S.~\textsc{McLaughlin}, T.~T. \textsc{Nguyen}, Q.~T.
  \textsc{Nguyen}, T.~\textsc{Nipkow}, S.~\textsc{Obua}, J.~\textsc{Pleso},
  J.~\textsc{Rute}, A.~\textsc{Solovyev}, T.~H.~A. \textsc{Ta}, N.~T.
  \textsc{Tran}, T.~D. \textsc{Trieu}, J.~\textsc{Urban}, K.~\textsc{Vu} \&
  R.~\textsc{Zumkeller} (2017), \enquote{A formal proof of the {{Kepler}}
  conjecture}. \emph{Forum of Mathematics. Pi}, \textbf{5}, \bblpp{} e2, 29.

\bibitem[{Iwasawa(1941)}]{Iwasawa-1941}
K.~\textsc{Iwasawa} (1941), \enquote{Über die {{Einfachheit}} der speziellen
  projektiven {{Gruppen}}}. \emph{Proceedings of the Imperial Academy},
  \textbf{17}~(3), \bblpp{} 57--59.

\bibitem[{Jacobson(1985)}]{Jacobson-1985}
N.~\textsc{Jacobson} (1985), \emph{Basic Algebra. {{I}}}, {W. H. Freeman and
  Company, New York}, second \bbledition{}.

\bibitem[{Jordan(1870)}]{Jordan-1870}
C.~\textsc{Jordan} (1870), \emph{{Traité des substitutions et des équations
  algébriques}}, {Gauthier-Villars}, {Paris}.

\bibitem[{Jordan(1872)}]{Jordan-1872a}
C.~\textsc{Jordan} (1872), \enquote{{Sur la limite de transitivité des groupes
  non alternés}}. \emph{Bulletin de la Société Mathématique de France},
  \textbf{1}, \bblpp{} 40--71.

\bibitem[{King(1981)}]{King-1981a}
O.~\textsc{King} (1981), \enquote{On some maximal subgroups of the classical
  groups}. \emph{Journal of Algebra}, \textbf{68}~(1), \bblpp{} 109--120.

\bibitem[{Liebeck \emph{\bbletal{}}(1987)Liebeck, Praeger \&
  Saxl}]{LiebeckPraegerSaxl-1987}
M.~W. \textsc{Liebeck}, C.~E. \textsc{Praeger} \& J.~\textsc{Saxl} (1987),
  \enquote{A classification of the maximal subgroups of the finite alternating
  and symmetric groups}. \emph{Journal of Algebra}, \textbf{111}~(2), \bblpp{}
  365--383.

\bibitem[{Massot \emph{\bbletal{}}(2022)Massot, {van Doorn} \&
  Nash}]{MassotvanDoornNash-2022}
P.~\textsc{Massot}, F.~\textsc{{van Doorn}} \& O.~\textsc{Nash} (2022),
  \enquote{Formalising the {$h$}-principle and sphere eversion}.
  \eprint{2210.07746}.

\bibitem[{Neumann(2006)}]{Neumann-2006}
P.~M. \textsc{Neumann} (2006), \enquote{The {{Concept}} of {{Primitivity}} in
  {{Group Theory}} and the {{Second Memoir}} of {{Galois}}}. \emph{Archive for
  History of Exact Sciences}, \textbf{60}~(4), \bblpp{} 379--429.

\bibitem[{Peano(1889)}]{Peano-1889}
G.~\textsc{Peano} (1889), \emph{{Arithmetices principia nova methoda
  exposita}}, {Fratelli Bocca Editrice}, {Torino}.

\bibitem[{Scholze(2022)}]{Scholze-2022}
P.~\textsc{Scholze} (2022), \enquote{Liquid {{Tensor Experiment}}}.
  \emph{Experimental Mathematics}, \textbf{31}~(2), \bblpp{} 349--354.

\bibitem[{{The mathlib Community}(2020)}]{ThemathlibCommunity-2020}
\textsc{{The mathlib Community}} (2020), \enquote{The lean mathematical
  library}. \emph{Proceedings of the 9th {{ACM SIGPLAN International
  Conference}} on {{Certified Programs}} and {{Proofs}}}, {{CPP}} 2020,
  \bblpp{} 367--381, {Association for Computing Machinery}, {New York, NY,
  USA}.

\bibitem[{Whitehead \& Russell(1927)}]{WhiteheadRussell-1927}
A.~N. \textsc{Whitehead} \& B.~\textsc{Russell} (1927), \emph{Principia
  Mathematica}, {Cambridge University Press}, {Cambridge [Eng.] New York}, 2d
  ed \bbledition{}.

\bibitem[{Wilson(2009)}]{Wilson-2009}
R.~A. \textsc{Wilson} (2009), \emph{The {{Finite Simple Groups}}}, Graduate
  {{Texts}} in {{Mathematics}}~\textbf{251}, {Springer London}, {London}.

\end{thebibliography}
\end{document}